\documentclass[]{amsart}

\usepackage{amsthm}
\usepackage{amsmath}
\usepackage{graphicx}
\usepackage{amsfonts}
\usepackage{amssymb}
\usepackage{latexsym}
\usepackage{color}
\usepackage{amscd}
\usepackage{hyperref}
\usepackage{mathpazo}
\usepackage{booktabs}
\usepackage{pstricks}
\usepackage{subfigure}

\everymath{\displaystyle}
\graphicspath{{img/}}

\newlength{\cellsize} 
\newsavebox{\cell}

\sbox{\cell}{\begin{picture}(18,18) \put(0,0){\line(1,0){18}}
\put(0,0){\line(0,1){18}} \put(18,0){\line(0,1){18}}
\put(0,18){\line(1,0){18}}
\end{picture}}

\newcommand\cellify[1]{\def\thearg{#1}\def\nothing{}%
\ifx\thearg\nothing \vrule width0pt height\cellsize depth0pt\else
\hbox to 0pt{\usebox{\cell} \hss}\fi%
\vbox to \cellsize{ \vss \hbox to \cellsize{\hss$#1$\hss} \vss}}

\newcommand\tableau[1]{\vtop{\let\\\cr
\baselineskip -16000pt \lineskiplimit 16000pt \lineskip 0pt
\ialign{&\cellify{##}\cr#1\crcr}}}

\newcommand\bas[1]{\omit \vbox to \cellsize{ \vss \hbox to \cellsize{\hss$#1$\hss} \vss}}

\newtheorem{theorem}{Theorem}[section]

\newtheorem{lemma}[theorem]{Lemma}

\theoremstyle{definition}
\newtheorem{definition}[theorem]{Definition}

\newtheorem{example}[theorem]{Example}

\begin{document}

\title{Rectification of Composition Tableaux}
\author{Melissa Bechard}
\maketitle

\begin{abstract}
In this paper we define an algorithm for rectifying one cell in a composition tableau. We then describe a generalization of this rectification process. The generalization is from one cell in the first column to any number of cells in the first column, provided they are bottom-justified. 
\end{abstract}

\section{Introduction}

\hspace{.1in} The study of symmetric functions has connections to many branches of mathematics, including algebraic geometry, group theory, representation theory, and lie algebras. Schur polynomials, a specific type of symmetric function that form an additive basis for the ring of symmetric polynomials, are the focus of this paper. There are several different bases for symmetric polynomials, and Schur polynomials are the most useful. This basis is described through fillings of partition diagrams. Other bases of symmetric polynomials are special cases of Schur polynomials. Schur polynomials provide information about the multiplicative structure of the cohomology ring of the Grassmannian. In representation theory, Schur polynomials are the characters of the general linear group. Combinatorially,  Schur polynomials appear in the Littlewood-Richardson Rule, evacuation, and the RSK algorithm ~\cite{skew, book}.

\hspace{.1in}Quasisymmetric functions are generalizations of symmetric functions. In a way, quasisymmetric functions can be thought of as functions that are between symmetric and non-symmetric. Quasisymmetric functions also relate to many algebraic structures. For example, the Hopf algebra of quasisymmetric functions is dual to the Hopf algebra of non-commutative symmetric functions. As with symmetric functions, there is a useful basis for quasisymmetric functions similar to Schur polynomials ~\cite{paper}. This basis is described through fillings of composition diagrams, called composition tableaux. 

\hspace{.1in} This paper focuses on the rectification of composition tableaux. There are several reasons rectification is useful. Rectification provides a way to prove when a skew Schur function can be written as a sum of Schur functions ~\cite{equality}. It is unknown when a skew quasisymmetric Schur function can be written as a sum of quasisymmetric Schur functions. We have reason to believe rectification can provide insight to this question. Rectification is also helpful in providing information on multiplication rules. We know how to multiply two Schur functions ~\cite{firstpaper}, as well as how to multiply a quasisymmetric Schur function with a Schur function ~\cite{triangle}. However, much is unknown such as: multiplication of two quasisymmetric Schur functions, multiplication of a skew quasisymmetric Schur function with a Schur function, and multiplication of two skew quasisymmetric Schur functions. In the case of the multiplication of two quasisymmetric Schur functions, rectification might be useful in keeping track of the sign change. Additionally, rectification is imperative for evacuation. Since evacuation is an algorithm defined around rectification of a cell, rectification is necessary in order to carry out this process. Evacuation is invertible, and in fact used to describe certain situations which occur when using the RSK algorithm ~\cite{evacuation, perms}. The RSK algorithm provides a bijection between $\mathbb{N}$-matrices and pairs of semistandard Young tableaux of the same shape $\lambda$. There is also a bijection between $\mathbb{N}$-matrices and composition tableaux which rearrange the same shape. It is unknown which matrices correspond to pairs of composition tableaux of the same shape. Rectification of a composition tableau will provide the foundation for evacuation of a composition tableau, since it should behave similarly to the evacuation of semistandard Young tableaux. 

\section{Background}

\hspace{.1in} A {\em symmetric polynomial} is a polynomial in n variables such that any permutation of the variables yields the original function. The ring of symmetric polynomials is denoted $Sym$. A {\em partition} of a positive integer $n$ is a way to write $n$ as a sum of positive integers. If two sums differ only in their order, then they are considered to be the same partition; if order matters then we call this a {\em composition}. For example, the partitions of $3$ are $(3)$, $(2$ , $1)$ and $(1$ , $1$ , $1)$. On the other hand, the compositions of $3$ are $(3)$, $(2$ , $1)$, $(1$ , $2)$, and $(1$ , $1$ , $1)$. We denote a partition as $\lambda$. A {\em Young diagram} is an array of left-justified cells that gives a visual representation of a partition $\lambda = (\lambda_1, \lambda_2, \lambda_3,...,\lambda_n)$, where $\lambda_i$ gives the number of cells in the $i^{th}$ row of the diagram. The cells are filled with positive integers so that the entries in each row are weakly increasing, while the entries in each column are strictly increasing. The result is a {\em semi-standard Young tableau}, often abbreviated {\em Young tableau} or SSYT.  See Figure \ref{fig:Young} for an example of a Young tableau of shape $\lambda = (4 ,3 ,3, 1)$. The frequency of each number in a tableau is the {\em weight}; refer to Figure \ref{fig: RSSYT}. We can look at a tableau and determine the weight, and from this we can find the associated Schur polynomial. A {\em standard Young tableau} is a Young tableau whose entries are numbered $1$ through $n$, where each number is used exactly once. Refer to Figure \ref{fig:standard}. We focus on {\em reverse semistandard Young tableau}, abbreviated RSSYT, which are analogous to semistandard Young tableaux; RSSYT make the proofs easier in this paper. For further details see \cite{book}.

\begin{figure}
\subfigure[Young tableau] {$ \tableau{2 & 2 & 4 & 5 \\ 4 & 5 & 7 \\ 5 & 6 & 8\\ 7 \\ & & &\\}$
\hspace{.15in}
\label{fig:Young}
}
\hspace{1 in}
\subfigure[standard Young tableau]{$\tableau{1 & 3 & 6 & 10\\ 2 & 5 & 8\\ 4 & 7 & 11\\ 9\\ & & & \\}$
\hspace{.6in}
\label{fig:standard}
}
\caption{A SSYT and a SYT of shape $4331$}
\label{testing}
\end{figure}

\vspace{.1in}
{\bf A reverse semistandard Young tableau} is a filling of a diagram with positive integers such that:

\begin{enumerate}
\item Row entries are weakly decreasing left to right,

\item Column entries are strictly decreasing from top to bottom
\end{enumerate}

\vspace{.1in}
\hspace{-.17in}See Figure \ref{fig: RSSYT} for an example of a reverse semistandard Young tableau.

\begin{figure}
$\tableau{9 & 8 & 6 & 4 & 2\\ 7 & 7 & 5 & 1 & 1\\ 5 & 4 & 2\\ 3 & 2 & 1\\ 1\\}$
\caption{RSSYT of weight (4, 3, 1, 2, 2, 1, 2, 1, 1)}
\label{fig: RSSYT}
\end{figure}

\vspace{.1in}
\hspace{.1in}{\em Schur polynomials} are symmetric polynomials that form an additive basis for the ring of symmetric polynomials. The Schur polynomials are used to record information about the multiplicative structure of groups and the classification of permutation groups. A Schur polynomial relates to the character of the general linear group of $n$ x $n$ matrices and is easily created from a tableau. Partitions are the indexing set for Schur polynomials. We define $S_{\lambda}(x_1, x_2, ..., x_n) = \sum x^T= \sum x_1^{t_1}x_2^{t_2}...x_n^{t_n}$, to be the sum over all SSYT $T$ of shape $\lambda$, where $T$ has weight: $(t_1, t_2,..., t_n)$.  Below is an example of a Schur polynomial, $s_{21}$, of shape $\lambda = (2,1)$ in three variables. To find $s_{21}$ we sum the the ways to fill a Young tableau of shape $\lambda$ with the integers $\{1, 2, 3 \}$.

\vspace{.2in}
\begin{example} A Schur polynomial, $s_{21}$, and the associated fillings. 

$$x_1 ^2x_2 = \tableau{1 & 1 \\ 2\\} \hspace{.2in},  \hspace{.1in} x_1 ^2 x_3 = \tableau{1 & 1 \\ 3} \hspace{.2in},\hspace{.1in}  x_2 ^2x_3 = \tableau{2 & 2 \\ 3}$$ 
$$x_1x_2^2 = \tableau{1 & 2 \\ 2} \hspace{.2in},\hspace{.1in}  x_1x_3^2 = \tableau{1 & 3 \\ 3} \hspace{.2in},\hspace{.1in}  x_2x_3^2 = \tableau{2 & 3 \\ 3}$$
$$x_1x_2x_3 = \tableau{1 & 2 \\ 3} \hspace{.2in}, \hspace{.1in} x_1x_2x_3 = \tableau{1 & 3 \\ 2}$$

$$s_{21}= (x_1 ^2x_2 + x_1 ^2 x_3 + x_2 ^2x_3) + (x_1x_2 ^2 + x_1x_3 ^2 + x_2x_3 ^2)  +(2x_1x_2x_3) $$

\end{example}

\vspace{.2in}
\hspace{.1in}Notice we have grouped $s_{21}$ in a very specific way, which is explained after a few more definitions. A {\em monomial symmetric polynomial}, $m_\alpha (x_1, x_2, ...,x_n)$,  is the sum of all monomials $x^\beta$, for all rearrangements $\beta$ of $\alpha$. This is best understood by an example. 

\vspace{.2in}
\begin{example} A monomial symmetric polynomial in three variables.
 $$m_{21} = (x_1 ^2x_2 + x_1 ^2 x_3 + x_2 ^2x_3) + (x_1x_2 ^2 + x_1x_3 ^2 + x_2x_3 ^2).$$
$$\mbox{Here $\alpha = (2$ , $1$) and $\beta = (2$ , $1$) , ($1$ , $2$).}$$
\end{example}

\vspace{.2in}
\hspace{.1in}First, notice $s_{21} = m_{21} + 2m_{111}$. We have also grouped $m_{21}$ in a very specific way. After a few more definitions this grouping will make more sense. A {\em composition} of $n$ is an ordered sequence of positive integers that sum to $n$. A {\em quasisymmetric function} is a bounded degree formal power series $F \in \mathbb{Q}[[x_1, x_2, ..,x_n]]$ such that for all $k$ the coefficient of $x_{i_1} ^{\alpha_1} x_{i_2} ^{\alpha_2} ...x_{i_k} ^{\alpha_k}$ is equal to the coefficient of $x_1 ^{\alpha_1} x_2 ^{\alpha_2}...x_k ^{\alpha_k}$ for all $i_1 < i_2 <...<i_k$ and for all compositions $(\alpha_1, \alpha_2, ..., \alpha_k)$~\cite{paper}. A quasisymmetric function with finitely many variables is a {\em quasisymmetric polynomial}.  We denote the ring of quasisymmetric polynomials as $Qsym$. The following are quasisymmetric polynomials: $f(x_1 , x_2 , x_3) = x_{1}^2 x_2 + x_{1}^2 x_3 + x_{2}^2 x_3$,  and $g(x_1, x_2, x_3) = x_{1}^2 x_{2} x_{3}^3$. Examples of  polynomials that are not quasisymmetric include $f(x_1, x_2)= x_1 ^2$ and $f(x_1,x_2, x_3)= x_1x_2 ^2 + x_1x_3 ^2$. Unlike symmetric polynomials where we need all possible permutations of all the exponents, quasisymmetric polynomials only need the permutations of the zero exponents keeping the nonzero exponents in the same order. The exponents in our first example are 2, 1, 0 and 2, 0, 1 and 0, 2, 1. Notice the $2$ and $1$ remained in the same order, while the zero exponent was permuted. Since the polynomial contains all permutations of the zero exponents, it is a quasisymmetric polynomial. The second example has exponents 2, 1, 3; since there are no zero exponents, no permutations are needed to make this a quasisymmetric polynomial. It is important to note that adding any permutation of a monomial which uses all the variables results in a quasisymmetric polynomial. For example, adding $x_1 x_2 ^3 x_3 ^2$  to $g$ gives $x_1^2x_2 x_3 ^3 + x_1 x_2 ^3 x_3 ^2$, which remains quasisymmetric. Also notice that quasisymmetric polynomials do not have to be symmetric, as $g$ is not, but all symmetric polynomials are quasisymmetric. Thus, $Sym \subseteq Qsym$. A {\em monomial quasisymmetric polynomial} $M_\alpha (x_1, x_2, ...,x_n)$, is the sum of all the monomials $x^\beta$, where $\beta$ runs over all the compositions of $\alpha$. $M_{21} = x_1 ^2x_2 + x_1 ^2 x_3 + x_2 ^2x_3$ is an example of a monomial quasisymmetric polynomial. Notice that $m_{21}$ can be written in terms of quasisymmetric monomials as $m_{21} = M_{21} + M_{12}$. We can also write $s_{21}$ as a sum of quasisymmetric monomials: $s_{21} = M_{21} + M_{12} + 2M_{111}$.

\vspace{.1in}
\hspace{.1in}A {\bf Composition Tableau (CT)} is a filling of a composition diagram with positive integers, satisfying the following properties:

\begin{enumerate}
\item Entries in the first column are strictly increasing from top to bottom,
\item Row entries are weakly decreasing from left to right,
\item Given any cell $a$ directly to the right of any cell $c$, and some cell $b$ that is below cell $a$ in the same column, but not necessarily directly below, if $a \leq b$ then $b > c$. (We think of empty cells as containing the entry $0$.) 
\end{enumerate}
$$\tableau{c & a \\ & \\ & b\\}$$

\begin{theorem}{\cite{paper}}
There exists a bijection, $\rho$, between composition tableaux and reverse semistandard Young tableaux. 

\end{theorem}

\hspace{.1in}Given a CT, $\rho$ arranges the column entries in decreasing order to produce a RSSYT. Given a RSSYT, $\rho^{-1}$ arranges the first column in increasing order from top to bottom. Then one column at a time, working from top to bottom, $\rho^{-1}$ places each entry in the highest possible position so that the row entries are weakly decreasing, producing a CT.  Below is an example of this bijection, where T is the RSSYT, and U is the CT. 

$${\bf U =} \hspace{.05in} \tableau{2 & 1\\ 3 & 2 & 2 & 2 &1\\ 4 & 4 & 4 &3\\ 6 & 5 & 5 &1\\ 7 & 7 & 3\\}  \hspace{.05in} \longrightarrow
\hspace{.05in} {\bf T=} \hspace{.05in} \tableau{7 & 7 & 5 & 3 & 1\\ 6 & 5 & 4 & 2\\ 4 & 4 & 3 & 1\\ 3 & 2 & 2\\ 2 & 1\\}$$    
\begin{center}
\begin{picture}(100,0)(0,0)
\setlength{\unitlength}{1pt}
\put(32,70){$\rho$}
\end{picture}
\end{center}

\hspace{.1in}Given two diagrams, one of shape $\lambda$ and one of shape $\mu$, if $\mu \subseteq \lambda$ then the {\em skew diagram}, $\lambda/\mu$, is the cells contained in $\lambda$ but not contained in $\mu$.  

\vspace{.1in}
\hspace{.1in}{\em Rectification} of a tableau is a procedure that gives a way to multiply Young tableaux. The process of rectifying a RSSYT is as follows: remove the highest entry in the first column, and consider this cell as an empty cell. Slide the larger of the two neighbors below and to the right of the empty cell into the empty cell. If the two neighbors have the same entry, then slide the lower entry into the empty cell. Whichever neighbor slid into the empty cell's spot, that neighbor's spot is now the new empty cell. Consider any empty cell as a $0$. Continue this process until there are no more neighbors to compare. We denote the process of rectifying a RSSYT by $\mu$. Below are $T$ and the rectification of $T$.

\begin{center}
\begin{picture}(400,280)(0,0)
\setlength{\unitlength}{1pt}
\put(-20,157){$T$ $=$}
\put(0,170){ \tableau{\rule{1.25em}{1.25em} & 7 & 5 & 3 & 1\\ 6 & 5 & 4 & 2\\ 4 & 4 & 3 & 1\\ 3 & 2 & 2\\ 2 & 1\\}}
\put(70,150){$\rightarrow$}
\put(100, 170){\tableau{7 & \rule{1.25em}{1.25em} & 5 & 3 & 1\\ 6 & 5 & 4 & 2\\ 4 & 4 & 3 & 1\\ 3 & 2 & 2\\ 2 & 1\\}}
\put(170, 150){$\rightarrow$}
\put(200, 170){ \tableau{7 & 5 & 5 & 3 & 1\\ 6 & \rule{1.25em}{1.25em} & 4 & 2\\ 4 & 4 & 3 & 1\\ 3 & 2 & 2\\ 2 & 1\\}}
\put(-20,50){$\rightarrow$}
\put(0, 70){\tableau{7 & 5 & 5 & 3 & 1\\ 6 & 4  & 4 & 2\\ 4 & \rule{1.25em}{1.25em} & 3 & 1\\ 3 & 2 & 2\\ 2 & 1\\}}
\put(70,50){$\rightarrow$}
\put(100, 70){\tableau{7 & 5 & 5 & 3 & 1\\ 6 & 4  & 4 & 2\\ 4 & 3 & \rule{1.25em}{1.25em} & 1\\ 3 & 2 & 2\\ 2 & 1\\}}
\put(170, 50){$\rightarrow$}
\put(200, 70){\tableau{7 & 5 & 5 & 3 & 1\\ 6 & 4  & 4 & 2\\ 4 & 3 & 2 & 1\\ 3 & 2 & \rule{1.25em}{1.25em}\\ 2 & 1\\}}
\put(25, -40){\mbox{Rectification of $T$ $=$}}
\put(130, -30){\tableau{7 & 5 & 5 & 3 & 1\\ 6 & 4  & 4 & 2\\ 4 & 3 & 2 & 1\\ 3 & 2 \\ 2 & 1\\}}
\end{picture}
\end{center}

\vspace{1.2in}
\hspace{.1in}{\em Evacuation} is a reverse sliding algorithm. The following is the algorithm for evacuation of a reverse semistandard Young tableau of shape $\lambda$: 
\begin{enumerate}
\item Remove the largest entry in column one.
\item Rectify the tableau.
\item Start a new RSSYT of the same shape $\lambda$.
\item In the corner that was removed after the rectification, fill in the new RSSYT with $(n -$ the number that was rectified$)$, where $n$ is the number of cells in the original tableau. 
\item Repeat until there are no more cells left to rectify in the original RSSYT.
\end{enumerate}

\vspace{.2in}
\begin{example} RSSYT of shape $\lambda = (3$, $2$, $1)$

\vspace{.1in}
\hspace{3.6in}\mbox{new RSSYT}
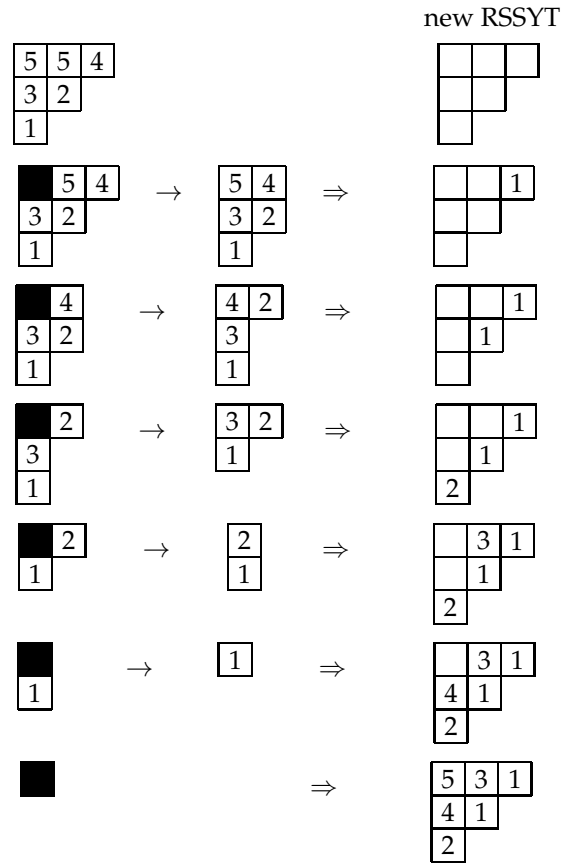
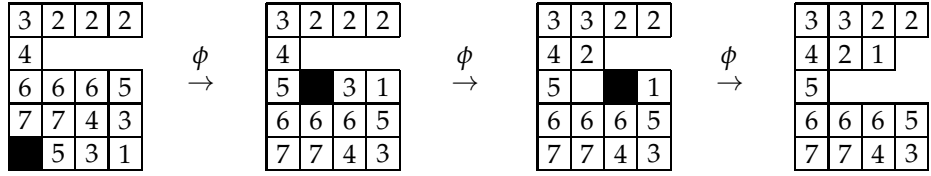
$$\tableau{5 & 5 & 4\\ 3 & 2\\ 1\\} \hspace{1.7in} \tableau{{} & {} & {}\\ {} & {} \\ {}\\}$$
$$ \tableau{ \rule{1.27em}{1.27em} & 5 & 4 \\ 3 & 2 \\ 1\\} \hspace{.15in} \rightarrow \hspace{.15in} \tableau{5 & 4\\ 3 & 2\\ 1\\} \hspace{.15in} \Rightarrow \hspace{.4in} \tableau{{} & {} & 1\\ {} & {}\\ {}\\}$$

$$\tableau{\rule{1.27em}{1.27em} & 4\\ 3 & 2\\ 1\\} \hspace{.25in} \rightarrow \hspace{.22in} \tableau{ 4 & 2\\ 3 \\ 1\\} \hspace{.18in} \Rightarrow \hspace{.4in} \tableau{{} & {} & 1\\ {} & 1\\ {}\\}$$

$$ \tableau{ \rule{1.27em}{1.27em} & 2\\ 3 \\ 1\\}\hspace{.25in} \rightarrow \hspace{.22in} \tableau{3 & 2\\ 1\\} \hspace{.18in} \Rightarrow \hspace{.4in} \tableau{{} & {} & 1\\ {} & 1\\ 2\\}$$

$$\tableau{ \rule{1.27em}{1.27em} & 2\\ 1\\} \hspace{.26in} \rightarrow \hspace{.26in} \tableau{2\\1\\}\hspace{.28in} \Rightarrow \hspace{.4in} \tableau{{} & 3 & 1\\ {} & 1\\ 2\\}$$

$$\tableau{\rule{1.27em}{1.27em} \\1\\} \hspace{.35in} \rightarrow \hspace{.3in} \tableau{1\\} \hspace{.31in} \Rightarrow \hspace{.42in} \tableau{{} & 3 & 1\\ 4 & 1\\ 2\\}$$

$$\tableau{\rule{1.27em}{1.27em}\\} \hspace{1.3in} \Rightarrow \hspace{.45in} \tableau{5 & 3 & 1\\ 4 & 1\\ 2\\}$$
\end{example}

\section{Rectifying a Composition Tableau}

{\bf $\phi$: Algorithm to rectify a composition tableau:}
\begin{enumerate}
\item Remove the largest entry in column one.
\item If there is a cell, $a$, directly to the right of the removed cell, move $a$ into the first column in such a way that the column remains strictly increasing. Moving $a$ into the $i^{th}$ row forces the first $(i-1)$ rows to shift up. Let $a$'s original cell be the new removed cell, and $a$'s original column and row be column $c$ and row $r$. 
\item Move the entry from column $(c +1)$ and row $r$ into column $c$, in the highest cell possible so that the rows remain weakly decreasing. This entry may bump any entry of smaller value. If an entry is bumped, that entry's spot is replaced by the new entry, and the bumped entry must find a new cell. 
\item Move the bumped entry to the next highest cell in that column, such that the corresponding row of that cell remains weakly decreasing. Bumped entries are also allowed to bump entries of smaller value. 
\item Repeat steps $3$ and $4$ for each subsequent column.
\item If there is no cell directly to the right of the removed cell then stop.

\end{enumerate}

\begin{center}
\begin{picture}(400,200)(0,0)
\setlength{\unitlength}{1pt}
\put(-25,100){ \tableau{3 & 2 & 2 & 2 \\ 4 \\ 6 & 6 & 6 & 5\\ 7 & 7 & 4 & 3 \\  \rule{1.25em}{1.25em} & 5 & 3 & 1\\}}
\put(45,80){$\rightarrow$}
\put(47, 90){$\phi$}
\put(75, 100){\tableau{3 & 2 & 2 & 2 \\ 4 \\ 5 & \rule{1.25em}{1.25em} & 3 & 1\\ 6 & 6 & 6 & 5 \\  7 & 7 & 4 & 3\\}}
\put(145, 80){$\rightarrow$}
\put(147,90){$\phi$}
\put(175, 100){ \tableau{3 & 3 & 2 & 2 \\ 4  & 2\\ 5 &  {} & \rule{1.25em}{1.25em} & 1\\ 6 & 6 & 6 & 5 \\  7 & 7 & 4 & 3\\}}
\put(245, 80){$\rightarrow$}
\put(247, 90){$\phi$}
\put(275, 100){\tableau{3 & 3 & 2 & 2\\ 4 & 2 & 1\\ 5 \\ 6 & 6 & 6 & 5\\ 7 & 7 & 4 & 3\\}}
\end{picture}
\end{center}

\begin{definition} Any entry that shifts into the previous column during rectification is called a {\em shifting entry}. 
\end{definition}

\begin{definition} In a RSSYT, an entry $f(c_{ij})$ has {\em diagonal dominance} if $f(c_{ij}) > f(c_{(i-1)(j+1)})$, where $f(c_{ij})$ is the entry in the $i^{th}$ column from the left and the $j^{th}$ row from the top.
\end{definition}

\begin{definition} Once the first diagonal dominant entry, $f(c_{ij})$,  is found as read from left to right, top to bottom, trace the southeast path of diagonally dominant entries. Let $S$ be the collection of the entries found on this path.
\end{definition}

Notice, this path is well-defined since there is only one southeast path of diagonally dominant entries starting at $f(c_{ij})$. We will see these entries are the shifting entries.

\begin{theorem}
The algorithm $\phi$ gives a rectification of composition tableaux for the largest entry in the first column, and commutes with the rectification of a RSSYT.
\end{theorem}

\vspace{.1in}
We prove this theorem by showing all the shifting entries from the RSSYT are positioned in the bottom row of the CT. We next prove that when we move $a$ (the entry to the right of the removed cell) into the first column, $a$ creates a new row, shifting higher rows up. We conclude by showing that moving the shifting entries, regardless of any bumping, results in a diagram of a CT. Our goal is to show that $\phi$ produces the same result as applying $\rho^{-1}$ to the rectified RSSYT. 

\vspace{.1in}
\begin{proof}
Given a RSSYT, where some cells might be empty. We think of empty cells as containing the entry $0$.  Let $f(k_{ij})$ be the largest entry from the left having diagonal dominance. We will only be concerned with $f(k_{ij})$ and the entries in $S$, since these are in fact the shifting entries. Let $f(m_{ij})$ be the largest entry in $S$. 

\vspace{.1in}
\hspace{.1in}All entries $f(c_{ij})$ larger than $f(k_{ij})$ satisfy $f(c_{ij}) \leq f(c_{(i-1)(j+1)})$, since these entries do not have diagonal dominance. So, there are $(j+1)$ many cells for these entries to be placed by $\rho^{-1}$ and only $(j-1)$ many entries have been placed, leaving two cells available. Since each entry must take the highest cell available, none of these entries will be placed in the bottom row. The next entry to be placed by $\rho^{-1}$ is $f(k_{ij})$. Since $f(k_{ij}) > f(c_{(i-1)(j+1)})$ by diagonal dominance, and all the entries bigger than $f(c_{(i-1)(j+1))})$ already have an entry to their right, $f(k_{ij})$ must be placed in the bottom row. We then take $f(m_{ij})$ and notice all entries strictly larger than $f(m_{ij})$ must have either not had diagonal dominance, which then we have the exact same argument as we did with $f(k_{ij})$, or there is at least one entry, $t$, that has diagonal dominance. In this case, $t$ must be higher in the RSSYT than $f(k_{ij})$. This is because $f(m_{ij})$ is the largest entry of $S$,  so there $t$ must be placed higher than $f(k_{ij})$ if it has diagonal dominance. Since $t$ is higher than $f(k_{ij})$ there must be some entry in $f(k_{ij})$'s column that $t$ sits next to. Therefore, $t$ can be placed next to this entry in the CT, and hence is not placed in the bottom row next to $f(k_{ij})$, since $t$ must take the highest cell available. The next entry to be placed by $\rho^{-1}$ is $f(m_{ij})$. Since $f(m_{ij})$ has diagonal dominance, we have the same argument as with $f(k_{ij})$, and $f(m_{ij})$ cannot find a position available except for in the bottom row next to $f(k_{ij})$. This argument is continued for each subsequent column. Thus, we have now shown shifting entries are in the bottom row of the CT. Columns that do not have an shifting entry must not have had any entry in the bottom row of those columns in the CT. This is because these columns have the property $f(c_{ij}) \leq f(c_{(i-1)(j+1)})$, for any $i$ and $j$. Thus, each of these entries has $(j+1)$ many cells where it may be placed. Since they must take the highest cell available none of these entries will be placed in the bottom row. Hence, only shifting entries are in the bottom row of the CT. 

\vspace{.1in}
\hspace{.1in}Next, we show moving $a$ into the first column shifts the rows higher than where $a$ was placed, up. After $f(k_{ij})$ has been placed in the $(i-1)^{th}$ column, causing the $(i-1)^{th}$ column to have at least one more entry than the $i^{th}$ column, all the entries larger than $f(k_{ij})$ in the $i^{th}$ column must be placed lower than $f(k_{ij})$. All of the entries smaller than $f(k_{ij})$ in the $i^{th}$ column are then placed above where $f(k_{ij})$ has been placed. This is because $\rho^{-1}$ requires entries to take the highest cell available, in decreasing order. Thus, no entry from the $i^{th}$ column has been placed next to $f(k_{ij})$. Therefore, moving $a$, or $f(k_{ij})$, into the first column, say in row $j$, causes the $(j-1)$ rows to shift up. 
Note if there is no entry $a$, then no entry had diagonal dominance, which means the remaining entries in the first column shift up in the RSSYT. Thus, only the largest entry in the first column has been removed. Therefore, we only need step $1$ for $\phi$ in this case. 

\vspace{.1in}
\hspace{.1in}We have shown only the shifting entries have moved columns by $\mu$, and are positioned in the bottom row of the composition tableau by $\rho^{-1}$. Only these entries in the bottom row of the CT shift columns by $\phi$. Notice there are no repeated entries in the columns. This is because each column has distinct entries in the reverse semi-standard Young tableau, and the only entries added to any column are those with diagonal dominance and weakly decreasing diagonal dominance. The rules of a reverse semi-standard Young tableau imply $f(k_{ij}) < f(k_{i(j-1)})$ and $f(k_{ij}) \leq f(k_{(i-1)j})$. We also know $f(k_{i(j-1)}) \leq f(k_{(i-1)j})$ since $f(k_{ij})$ is the largest entry in column $i$ to have diagonal dominance. This gives us $f(k_{ij})<f(k_{i(j-1)}) \leq f(k_{(i-1)j})$. So, $f(k_{ij}) <f(k_{(i-1)j})$. We also know that $f(k_{ij}) > f(k_{(i-1)(j+1)})$ by diagonal dominance. Since each column has distinct entries in RSSYT, and $f(k_{ij}) <f(k_{(i-1)j})$ and $f(k_{ij}) > f(k_{(i-1)(j+1)})$, this means $f(k_{ij})$ cannot be the same as any entry in the $(i-1)^{th}$ column. This argument holds for all of the shifting entries. So, no new entry is added to any column. Hence, this process does indeed give the same tableau. 

\vspace{.1in}
\hspace{.1in}Lastly, we show $\phi$ produces a composition tableau. Notice property $1$ of a composition tableau is satisfied since $a$ is moved into the first column by $\phi$, in the highest spot allowing the column to be strictly increasing. Each shifting entry is moved into the previous column in the highest cell, such that the row into which they are placed remains weakly decreasing, satisfying property $2$ of a composition tableau. The map $\rho^{-1}$ fixes each of the columns. This is because $\rho^{-1}$ is defined to re-order each column keeping the associated entries within that column. This means bumping in $\phi$ obeys the insertion rule, and property $3$ must hold. Therefore, we do in fact have a valid map from a composition tableau to a composition tableau. 

\end{proof}

\vspace{-.5in}
\hspace{.1in}Below is an example of a composition tableau, $U$, and its associated reverse semi-standard Young tableau, $T$ followed by both of their rectifications. We can verify the rectification of $U$ using the defined process of rectifying $T$, $\mu$, and then mapping this tableau back to its associated composition tableau using $\rho^{-1}$. We can see that $\phi$ maps $U$ to the same tableau as rectifying $T$ and then applying $\rho^{-1}$.

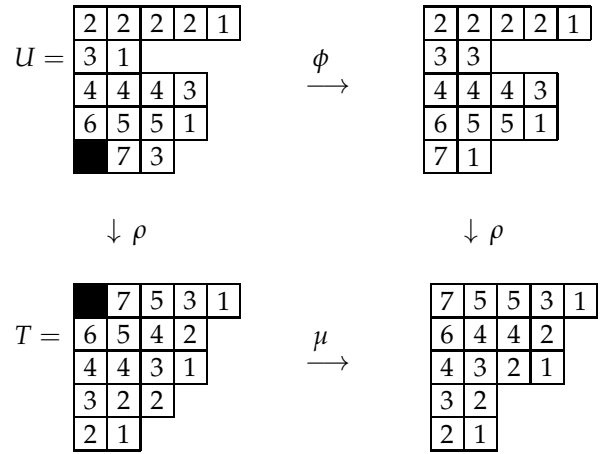
\begin{center}
\begin{picture}(280,280)(0,0)
\setlength{\unitlength}{1pt}
\put(-20,160){$U=$}
\put(0,170){ \tableau{2 & 2 & 2 & 2 & 1\\ 3 & 1\\ 4 & 4 & 4 &3\\ 6 & 5 & 5 &1\\ \rule{1.27em}{1.27em} & 7 & 3\\}}
\put(93,160){$\phi$}
\put(90,150){$\longrightarrow$}
\put(135, 170){\tableau{2 & 2 & 2 & 2 & 1\\ 3 & 3\\ 4 & 4 & 4 &3\\ 6 & 5 & 5 &1\\ 7 & 1\\}}
\put(25, 95){$\rho$}
\put(15, 95){$\downarrow$}
\put(0, 65){ \tableau{\rule{1.25em}{1.25em} & 7 & 5 & 3 &1\\ 6 & 5 & 4 & 2\\ 4 & 4 & 3 & 1\\ 3 & 2 & 2\\ 2 & 1\\} }
\put(-20, 55){$T=$}
\put(93,55){$\mu$}
\put(90,45){$\longrightarrow$}
\put(135,65){ \tableau{7 & 5 & 5 & 3 & 1\\ 6 & 4 & 4 & 2\\ 4 & 3 & 2 & 1\\ 3 & 2\\ 2 & 1\\}}
\put(150 ,95){$\downarrow$}
\put(160,95){$\rho$}

\end{picture}
\end{center}

\section{A more generalized rectification of Composition Tableaux}

\vspace{.1in}
The following is an algorithm, $\phi$, to rectify $k$ cells in the first column of a composition tableau that are adjacent and bottom-justified, followed by an example to illustrate the algorithm.

{\bf Algorithm $\phi$:}
\begin{enumerate}
\item Remove the largest $k$ entries in column one.
\item Swap all entries directly right of the $k$-removed cells with the $k$-removed cells.
\item Reorder the rows so that the first column entries are strictly increasing.
\item Start with the largest entry in column $c$ to the right of a removed box, and insert this entry into column $(c-1)$ into the highest cell possible so that the rows remain weakly decreasing. This entry may bump any entry of smaller value. Bumped entries are moved into the next highest cell in that column so that the rows remain weakly decreasing (bumping if necessary). Repeat for the next largest entry in column $c$ to the right of a removed box. Continue until there are no more entries to the right of the removed boxes in column $c$. The cells  of the entries from column $c$ that were inserted into column $(c-1)$ are thought of as the new removed boxes. 
\item Repeat step $4$ for each subsequent column.
\item If there are no cells directly right of the removed cells then stop.
\end{enumerate}

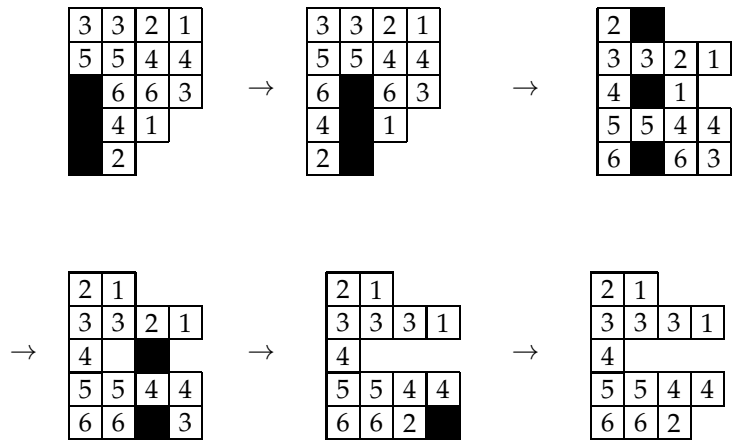
\begin{center}
\begin{picture}(400,300)(0,0)
\setlength{\unitlength}{1pt}
\put(0,170){ \tableau{3 & 3 & 2 & 1 \\ 5 & 5 & 4 & 4 \\ \rule{1.27em}{1.27em} & 6 & 6 & 3 \\ \rule{1.27em}{1.27em} & 4 & 1 \\  \rule{1.27em}{1.27em} & 2 \\}}
\put(70,150){$\rightarrow$}
\put(90,170){ \tableau{3 & 3 & 2 & 1 \\ 5 & 5 & 4 & 4 \\6 &  \rule{1.27em}{1.27em} & 6 & 3 \\ 4 &  \rule{1.27em}{1.27em} & 1 \\ 2 & \rule{1.27em}{1.27em} \\}}
\put(170, 150){$\rightarrow$}
\put(200, 170){ \tableau{2 & \rule{1.25em}{1.25em}\\ 3 & 3 & 2 &1 \\ 4 & \rule{1.25em}{1.25em} & 1 \\ 5 & 5 & 4 & 4 \\ 6 &  \rule{1.27em}{1.27em} & 6 & 3 \\}}
\put(-20, 50){$\rightarrow$}
\put(0, 70){ \tableau{2 & 1 \\ 3 & 3 & 2 & 1 \\ 4 &  & \rule{1.25em}{1.25em} \\ 5 & 5 & 4 & 4 \\ 6 & 6 & \rule{1.25em}{1.25em}& 3 \\}}
\put(70, 50){$\rightarrow$}
\put(100, 70){\tableau{2 & 1\\3 & 3 & 3 & 1\\ 4 \\5 & 5 & 4 & 4\\6 & 6 & 2 & \rule{1.27em}{1.27em}\\}}
\put(170, 50){$\rightarrow$}
\put(200, 70){\tableau{2 & 1\\3 & 3 & 3 & 1\\ 4 \\5 & 5 & 4 & 4\\6 & 6 & 2\\}}

\end{picture}
\end{center}

\begin{lemma}
Rectifying cells $1$ through $k$ in a RSSYT, where cell $1$ contains the largest entry, for each $c$, the shifting entry in column $c$ during the $n^{th}$ rectification is the $n^{th}$ largest shifting entry in column $c$.
\end{lemma}

\hspace{.1in} To prove this lemma, we show given a shifting entry, $f(c_{ij})$, of the $i^{th}$ column, that moves during the $m^{th}$ rectification, when we rectify the $(m+1)^{th}$ term, there must be a shifting entry higher than $f(c_{ij})$ since $f(c_{ij})$ will be compared to those entries in the $i^{th}$ column. Thus, the shifting entries during the $n^{th}$ rectification are the $ n^{th}$ largest shifting entries in each column.

\hspace{.1in}We can order the entries of a RSSYT in such a way that the entries that move columns during rectification are apparent. We call this ordering {\em eviction}. Remove the largest $k$-many entries in the first column. Starting with the second column, these entries remain in strictly decreasing order. We then align the entries remaining in the first column so that they are placed as high as possible on the left side of the second column with the rows weakly decreasing. The entries of the second column that have no entry of the first column beside them are removed and the remaining entries are used to align among the third column. We continue this process through each of the columns. Below is an example of a RSSYT applying the eviction ordering.

\begin{example}$$\tableau{\rule{1.25em}{1.25em} & 8 & 6\\ \rule{1.25em}{1.25em} & 5 & 3\\\rule{1.27em}{1.27em} & 4 & 1\\5 & 2\\3 \\} \hspace{.2in} \rightarrow \hspace{.2in} \tableau{& 8\\ 5 & 5\\ & 4\\ 3 & 2} \hspace{.2in} \Rightarrow \hspace{.2in} \mbox{$8$ and $4$ are shifting entries from column $2$}$$
 \mbox{$5$ and $2$ are now aligned among the third column}
$$\hspace{.575in} \rightarrow \hspace{.2in} \tableau{ & 6\\ 5 & 3\\ 2 & 1\\} \hspace{.2in} \Rightarrow \hspace{.2in} \mbox{$6$ is the only shifting entry of column $3$}$$
 \mbox{if there was fourth column then $3$ and $1$ would be aligned among that column}
\end{example}

\begin{lemma} The entries of the $i^{th}$ column in a RSSYT do not have an entry from the $(i-1)^{th}$ column beside them for eviction if and only if the entries are shifting entries of the $i^{th}$ column. 
\end{lemma}

\hspace{.1in} To prove the reverse implication of this lemma we use the fact that a shifting entry has diagonal dominance, and thus, when aligning any two columns for eviction, there can never be any entry beside a shifting entry. For the forward direction, we use the comparisons made in rectification to show there cannot be a shifting entry abbove or below the entry, $f(c_{ij}0$, where $f(c_{ij})$ is the only such entry in the $i^{th}$ column that has no entry beside it in the $(i-1)^{th}$ column for eviction. We then conclude $f(c_{ij})$ must be a shifting entry because of these comparisons. We then prove this process inductively, using a similar argument. 

\begin{lemma} The shifting entries of the $i^{th}$ column, found via eviction, are exactly the entries in the same rows as the entries we want to rectify in the composition tableau.
\end{lemma}
\begin{proof} Consider the shifting entry $f(c_{2j})$. Since $f(c_{2j})$ is a shifting entry that means there is no entry of the $1^{st}$ column beside it during eviction. Notice, $\phi^{-1}$ inserts all of those entries in column $2$ that are larger than $f(c_{2j})$ prior to $f(c_{2j})$. When we insert $f(c_{2j})$, the remaining entries from those rectified have been ordered as high as possible by the eviction process, and thus since $f(c_{2j})$ does not have an entry beside it during eviction, $f(c_{2j})$ must be placed among one of those entries that will be rectified. Similarly, given a shifting entry, $f(c_{ij})$, of the $i^{th}$ column, those entries of the $(i-1)^{th}$ column have been positioned as high as possible in the eviction process, and since $f(c_{ij})$ has no entry from the $(i-1)^{th}$ column beside it, $f(c_{ij})$ must be placed beside one of the shifting entries of the $(i-1)^{th}$ column which is placed in the same row as one of the entries to be rectified in the CT.  
\end{proof}

\begin{theorem} Using $\phi$ gives a rectification of composition tableaux for any number of entries in the first column, and commutes with the rectification of a RSSYT. 
\begin{proof}
By Lemma $4.1$ we know entries that shift columns during the RSSYT rectification are associated in decreasing order with the entries that are rectified. We also know by Lemma $4.2$ that eviction gives a natural way to find these entries. We then know by Lemma $4.3$ these same entries are the entries in the same rows as the entries of the CT that we are rectifying. We can then use $\phi$, which is defined to shift each of these entries exactly one column to the left, as in the RSSYT, which preserves each of the columns of the RSSYT and CT. Thus, all that is left to check is that $\phi$ produces a CT. 

\vspace{.1in}
\hspace{.1in}Step $2$ of the algorithm requires the first column to remain strictly decreasing. The only concern is, could there be duplicate entries in this column when the second column entries to the right of the removed cells are moved into the first column? In fact this cannot happen. Initially, the entries in each column are distinct. When the entries are inserted during $\rho$ they are placed so that they are as high as possible with the rows remaining weakly decreasing. This requirement ensures that an entry above those rectified, which are the only remaining entries of the first column, could not have a duplicate entry from the second column that is moved into the first column. This is because, if an entry, $b$, of the second column had the same value as one of the entries of the first column, $a$,  it must have been placed next to a, since that is as high as it could be placed. The only way an $b$ would not be alongside $a$ is if some other larger entry had been placed prior to $b$ alongside $a$. However, an entry cannot be placed alongside $a$ unless it is of equal or lesser value. Thus, an entry of value $a$ in the second column would be placed directly next to $a$ and therefore, would not shift columns during $\phi$. 

\vspace{.1in}
\hspace{.1in}Step $3$ of the algorithm rearranges each column, by possibly adding some entries from the next column, and placing them in decreasing order. This placement follows directly from the mapping of $\rho^{-1}$. Since, one column at a time, $\rho^{-1}$ places each entry in the highest possible position so that the row entries are weakly decreasing, step $3$ of $\phi$ satisfies the conditions of a composition tableau. Thus, again our only concern is if there are duplicate entries in any of the columns, since these columns must have distinct entries. We actually have the same argument as before. Any entry of the $i^{th}$ column that has an entry of the same value in $(i-1)^{th}$ column is placed next to that entry in the $(i-1)^{th}$ column, since that is as high as it can be placed. No larger entry could have taken its spot, since a larger entry will violate the weakly decreasing requirement of each row. So, we have both property $1$ and property $2$ of a composition tableau satisfied. Thus, steps $1$ and $2$ of the algorithm obey the insertion rule of $\phi$ and force property $3$ of a composition tableau to hold as well. Therefore, our algorithm does give a map from a composition tableau to a composition tableau that corresponds to the given RSSYT. 
\end{proof}
\end{theorem}

\section{Future directions/conclusion}
\hspace{.1in}There are several directions we can go to further our generalization of rectifying composition tableau. We have looked at rectifying an entry that is not in the first column, a column other than the first column, as well as a row. Several patterns have emerged, but we are still looking for a concrete proof. We have reason to believe if we knew something about rectifying punctured diagrams, then we may have a better approach to show these generalizations. Another direction would be to look into fixing a reading order in the reverse semistandard Young tableaux in order to help with rectification of punctured diagrams. 

\bibliographystyle{plain}
\bibliography{paper2}

\end{document}